\documentclass[11pt]{amsart}
\usepackage[margin=1in]{geometry}
\usepackage{amsmath,amssymb,amsthm,mathtools}
\usepackage{microtype}
\usepackage{enumitem}
\usepackage[hidelinks]{hyperref}
\usepackage{fancyhdr}

\newtheorem{theorem}{Theorem}[section]
\newtheorem{proposition}[theorem]{Proposition}
\newtheorem{lemma}[theorem]{Lemma}

\theoremstyle{definition}

\newcommand{\Vol}{\operatorname{Vol}}
\newcommand{\inj}{\operatorname{inj}}
\newcommand{\diam}{\operatorname{diam}}
\newcommand{\Hess}{\operatorname{Hess}}
\newcommand{\dist}{\operatorname{dist}}

\title{Almost Positive Sectional Curvature with Bounded Geometry on $S^2\times S^2$}
\author{Yuhang Liu\\Xi'an Jiaotong-Liverpool university, Department of Pure Mathematics, \\111 Ren'ai Road, Suzhou Industrial Park, Suzhou, Jiangsu Province, China\\yuhang.liu02@xjtlu.edu.cn\\
+86-0512-89167161}
\date{2026-8-21}

\begin{document}
\maketitle

\begin{abstract}
Let $M=S^2\times S^2$.  We prove that there is a fixed smooth open ball $U\subset M$ such that, for every $\varepsilon>0$, one can find a smooth metric $g_\varepsilon$ and a closed set $S_\varepsilon$ with the following properties: $\sec_{g_\varepsilon}>0$ on $M\setminus S_\varepsilon$; the relative volume of $S_\varepsilon$ is less than $\varepsilon$; every sectional curvature on $U$ is greater than $1$; $-\varepsilon<\sec_{g_\varepsilon}\leq\Lambda$ for one constant $\Lambda$ independent of $\varepsilon$; the volume and diameter have uniform positive lower and upper bounds; and the injectivity radius has a uniform positive lower bound.  In fact, $U\cap S_\varepsilon=\varnothing$.

The construction starts from the product of the unit round metrics.  A small conformal factor, obtained by smoothing the absolute value of the height function on each sphere, has negative definite Hessian away from two thin equatorial bands.  It makes every formerly flat mixed plane strictly positive outside a set of arbitrarily small relative volume, while the negative curvature created in the bands is only $O(\varepsilon)$.  A second conformal deformation supported in a very small region produces a core on which all sectional curvatures are greater than $1$; its transition has Hessian bounded above by $O(\varepsilon)$.  Finally, a diffeomorphism compresses the fixed topological ball $U$ into that core.  Pullback preserves all intrinsic bounded-geometry estimates. The main content of this paper is generated by ChatGPT 5.6 and verified by the author.
\end{abstract}

\section{Introduction and Statement}
The main purpose of this paper is to try to produce a metric on $S^2\times S^2$ with quasi-positive sectional curvature, which is not successful. We note that recently Brendle and Hung constructed metrics with positive sectional curvature on $S^2\times S^2$, resolving the Hopf conjecture \cite{Brendle}. However, since the metrics constructed in this paper is relatively simple and explicit, we post it as a note of independent interest. Main techniques used in this paper can be found in standard textbooks like \cite{Klingenberg, Petersen}.

\textbf{Acknowledgement:} the author thanks Xi'an Jiaotong-Liverpool University for providing financial and academic support. He also thanks Zipei Nie for providing access to ChatGPT and insights about AI usage. His gratitude also goes to his daughter and son who were born last year, and all of his family members who have always been supportive during his academic career.\\

Let $h$ be the unit round metric on $S^2$, and put
\[
 g_0=h\oplus h
 \quad\text{on}\quad
 M=S^2\times S^2.
\]
An open ball below means a subset diffeomorphic to the standard open $4$-ball; it is fixed as a subset of the smooth manifold and is not required to have a fixed intrinsic radius for the varying metrics.

\begin{theorem}\label{thm:main}
There exist a smooth open ball $U\subset M$ and constants
\[
 0<v_-\le v_+<\infty,\qquad
 0<d_-\le d_+<\infty,\qquad
 i_0>0,\qquad \Lambda<\infty,
\]
such that, for every $\varepsilon>0$, there are a smooth Riemannian metric $g_\varepsilon$ on $M$ and a closed subset $S_\varepsilon\subset M$ satisfying
\begin{enumerate}[label=\textup{(\arabic*)}]
 \item $\sec_{g_\varepsilon}>0$ on $M\setminus S_\varepsilon$;
 \item
 \[
  \frac{\Vol(S_\varepsilon,g_\varepsilon)}
       {\Vol(M,g_\varepsilon)}<\varepsilon;
 \]
 \item $U\cap S_\varepsilon=\varnothing$, and every two-plane based at a point of $U$ has sectional curvature greater than $1$;
 \item
 \[
  -\varepsilon<\sec_{g_\varepsilon}\le\Lambda
  \quad\text{on }M;
 \]
 \item
 \[
  v_-\le\Vol(M,g_\varepsilon)\le v_+,
  \qquad
  d_-\le\diam(M,g_\varepsilon)\le d_+;
 \]
 \item $\inj(M,g_\varepsilon)\ge i_0$.
\end{enumerate}
\end{theorem}

The freedom removed from the earlier one-region formulation is crucial.  The exceptional set constructed here is distributed along two thin equatorial slabs, together with one small transition annulus.  Its volume tends to zero, but it is not forced to have small diameter.

\section{The product metric and conformal curvature}

For a two-plane $P\subset T_{(x,y)}M$, choose a $g_0$-orthonormal basis
$X=(X_1,X_2)$, $Y=(Y_1,Y_2)$.  The product curvature tensor gives
\begin{equation}\label{eq:product-curvature}
 \sec_{g_0}(P)
 =|X_1\wedge Y_1|_h^2+|X_2\wedge Y_2|_h^2.
\end{equation}
Hence
\begin{equation}\label{eq:product-bounds}
 0\le\sec_{g_0}\le1.
\end{equation}
The zero-curvature planes are precisely the mixed planes.

We shall repeatedly use the conformal curvature formula.  If
\[
 \widetilde g=e^{2f}g
\]
and $e_1,e_2$ are $g$-orthonormal vectors spanning $P$, then
\begin{equation}\label{eq:conformal-curvature}
\begin{split}
 e^{2f}\sec_{\widetilde g}(P)
 ={}&\sec_g(P)-\Hess_g f(e_1,e_1)-\Hess_g f(e_2,e_2)\\
 &+df(e_1)^2+df(e_2)^2-|\nabla^g f|_g^2.
\end{split}
\end{equation}
In particular,
\begin{equation}\label{eq:gradient-worst}
 df(e_1)^2+df(e_2)^2-|\nabla f|^2\ge-|\nabla f|^2.
\end{equation}

\section{An almost concave function on the sphere}

Let $z:S^2\to[-1,1]$ be the height function.  For the unit round metric,
\begin{equation}\label{eq:height}
 |\nabla z|^2=1-z^2,
 \qquad
 \Hess_h z=-z h.
\end{equation}

Choose once and for all an even nonnegative function
$\beta\in C_c^\infty((-1,1))$ with $\int_{\mathbb R}\beta=1$, and put
\[
 \beta_\delta(t)=\delta^{-1}\beta(t/\delta),
 \qquad
 \psi_\delta=|\cdot|*\beta_\delta,
 \qquad 0<\delta<\frac14.
\]
Let
\[
 C_\beta=2\|\beta\|_\infty.
\]

\begin{lemma}\label{lem:smoothing}
The function $\psi_\delta$ is smooth, even and convex, and it satisfies
\begin{enumerate}[label=\textup{(\roman*)}]
 \item $\psi_\delta(t)=|t|$ whenever $|t|\ge\delta$;
 \item $|\psi_\delta'|\le1$, $t\psi_\delta'(t)\ge0$, and
 \[
 0\le\psi_\delta''(t)\le\frac{C_\beta}{\delta};
 \]
 \item $0\le\psi_\delta(t)\le1+\delta$ for $|t|\le1$.
\end{enumerate}
If $\phi_\delta=\psi_\delta\circ z$, then
\begin{equation}\label{eq:hess-phi}
 \Hess_h\phi_\delta
 =\psi_\delta''(z)\,dz\otimes dz-z\psi_\delta'(z)h.
\end{equation}
Consequently, as quadratic forms,
\begin{equation}\label{eq:hess-global-bounds}
 -h\le\Hess_h\phi_\delta
 \le\frac{C_\beta}{\delta}h,
 \qquad
 |\nabla\phi_\delta|\le1,
\end{equation}
and on $\{|z|>\delta\}$,
\begin{equation}\label{eq:hess-concave}
 \Hess_h\phi_\delta=-|z|h\le-\delta h.
\end{equation}
\end{lemma}

\begin{proof}
Convolution preserves smoothness, evenness and convexity.  If $t\ge\delta$, then $t-s\ge0$ for every $s$ in the support of $\beta_\delta$, and the evenness of $\beta_\delta$ gives
\[
 \psi_\delta(t)=\int(t-s)\beta_\delta(s)\,ds=t.
\]
The case $t\le-\delta$ is identical.  Since the distributional second derivative of $|t|$ is twice the Dirac mass at the origin,
\[
 \psi_\delta''=2\beta_\delta.
\]
The remaining one-dimensional assertions follow immediately.  Formula
\eqref{eq:hess-phi} follows from the chain rule and \eqref{eq:height}.
Because $z\psi_\delta'(z)\ge0$, the second term in
\eqref{eq:hess-phi} is nonpositive.  The upper bound follows from
$dz\otimes dz\le h$, and the lower bound follows from
$|z\psi_\delta'(z)|\le1$.  On $|z|>\delta$ one has
$\psi_\delta(z)=|z|$, which gives \eqref{eq:hess-concave}.
\end{proof}

For $E_\delta=\{|z|\le\delta\}\subset S^2$, spherical coordinates give
\begin{equation}\label{eq:band-area}
 \frac{\operatorname{Area}(E_\delta,h)}
      {\operatorname{Area}(S^2,h)}=\delta.
\end{equation}
Define the two equatorial slabs
\begin{equation}\label{eq:S0}
 \Sigma_\delta
 =(E_\delta\times S^2)\cup(S^2\times E_\delta).
\end{equation}
Then
\begin{equation}\label{eq:S0-volume}
 \frac{\Vol(\Sigma_\delta,g_0)}{\Vol(M,g_0)}
 =2\delta-\delta^2<2\delta.
\end{equation}

For $a>0$, set
\begin{equation}\label{eq:u-def}
 u_{a,\delta}(x,y)
 =a\bigl(\phi_\delta(x)+\phi_\delta(y)\bigr).
\end{equation}
Lemma~\ref{lem:smoothing} gives
\begin{equation}\label{eq:u-global}
 -ag_0\le\Hess_{g_0}u_{a,\delta}
 \le\frac{aC_\beta}{\delta}g_0,
 \qquad
 |\nabla u_{a,\delta}|^2\le2a^2,
\end{equation}
and, on $M\setminus\Sigma_\delta$,
\begin{equation}\label{eq:u-concave}
 \Hess_{g_0}u_{a,\delta}\le-a\delta g_0.
\end{equation}

\begin{proposition}\label{prop:base-conformal}
If $0<a<\delta$, then the metric
\[
 \bar g_{a,\delta}=e^{2u_{a,\delta}}g_0
\]
has strictly positive sectional curvature on
$M\setminus\Sigma_\delta$.  Globally,
\begin{equation}\label{eq:base-lower}
 \sec_{\bar g_{a,\delta}}
 \ge-\frac{2aC_\beta}{\delta}-2a^2,
\end{equation}
and
\begin{equation}\label{eq:base-upper}
 \sec_{\bar g_{a,\delta}}\le1+2a+2a^2.
\end{equation}
\end{proposition}

\begin{proof}
Apply \eqref{eq:conformal-curvature} with $g=g_0$ and
$f=u_{a,\delta}$.  On $M\setminus\Sigma_\delta$,
\eqref{eq:u-concave}, \eqref{eq:gradient-worst}, and
\eqref{eq:product-bounds} give
\[
 e^{2u_{a,\delta}}\sec_{\bar g_{a,\delta}}(P)
 \ge2a\delta-2a^2>0.
\]
The global lower bound follows from the upper Hessian estimate in
\eqref{eq:u-global}.  For the upper bound, use
$\Hess u_{a,\delta}\ge-ag_0$ and
\[
 du(e_1)^2+du(e_2)^2-|\nabla u|^2\le|\nabla u|^2.
\]
Finally $u_{a,\delta}\ge0$, so the factor $e^{-2u_{a,\delta}}$ never enlarges a positive upper bound or a negative lower bound in the estimates above.
\end{proof}

\section{A bounded-curvature positive needle}

We next construct a conformal factor whose Hessian is strongly negative on a very small core, while its positive Hessian is arbitrarily small everywhere.

\begin{lemma}[local needle]\label{lem:needle}
Let $(N,g)$ be a smooth closed Riemannian manifold and let $q\in N$.  There are constants
\[
 R_*>0,\qquad c_*>0,\qquad C_*>1,
\]
depending only on $(N,g,q)$, with the following property.  For every
$0<\sigma\le1$ and every $0<R\le R_*$ there is a smooth function
$w=w_{\sigma,R}:N\to[0,\infty)$ such that
\begin{enumerate}[label=\textup{(\roman*)}]
 \item $w=0$ outside $B_g(q,R)$ and in a neighborhood of its boundary;
 \item
 \[
  -C_*g\le\Hess_g w\le\sigma g;
 \]
 \item for some $\rho\ge c_*\sigma R$,
 \[
  \Hess_g w\le-3g
  \quad\text{on }B_g(q,\rho);
 \]
 \item
 \[
  |dw|_g\le C_*\sigma R,
  \qquad
  0\le w\le C_*\sigma R^2.
 \]
\end{enumerate}
\end{lemma}

\begin{proof}
Let $r=\dist_g(q,\cdot)$.  After decreasing $R_*$, geodesic polar coordinates are smooth on $0<r<R_*$ and
\begin{equation}\label{eq:hess-r}
 \Hess_g r
 =\frac1r(g-dr^2)+E_r,
 \qquad
 |E_r|_g\le C r,
\end{equation}
for a fixed constant $C$.  Equivalently,
$\Hess_g(r^2/2)=g+O(r^2)$.

We construct a one-variable function $F$.  Put $A=4$.  Choose smooth functions
$\alpha:[0,\infty)\to[0,1]$ and
$\gamma\in C_c^\infty((0,1))$ such that
\[
 \alpha=1\ \text{on }[0,1],\qquad
 \alpha=0\ \text{on }[2,\infty),\qquad
 \gamma\ge0,\qquad \int_0^1\gamma=1.
\]
Choose a sufficiently small fixed constant $c>0$ and put
$\ell=c\sigma R$.  Define
\[
 n(r)=-A\alpha(r/\ell),
 \qquad
 D=-\int_0^R n(s)\,ds,
\]
and put a positive bump of total mass $D$ in the interval
$(3\ell,R-\ell)$:
\[
 p(r)=
 \frac{D}{R-4\ell}
 \gamma\!\left(\frac{r-3\ell}{R-4\ell}\right)
\]
there, and $p=0$ elsewhere.  Taking $c$ small, independently of
$\sigma$ and $R$, ensures
\[
 0\le p\le\frac{\sigma}{4},
 \qquad 4\ell<R.
\]
Set
\[
 q_1=n+p,
 \qquad
 Q(r)=\int_0^r q_1(s)\,ds.
\]
The negative bump occurs before the positive one, the two total masses cancel, and therefore
\begin{equation}\label{eq:Q-properties}
 Q\le0,\qquad Q=0\text{ near }R,
 \qquad |Q|\le C_1\sigma R.
\end{equation}
Moreover,
\begin{equation}\label{eq:q-properties}
 -A\le q_1\le\frac{\sigma}{4},
 \qquad
 \frac{Q(r)}r\ge-C_2A.
\end{equation}
Define
\[
 F(r)=-\int_r^R Q(s)\,ds.
\]
Then $F\ge0$, $F'=Q$, $F''=q_1$, and $F=0$ near $R$.  On
$0\le r\le\ell$ one has
\[
 F(r)=F(0)-\frac A2r^2.
\]
The radial function $w=F(r)$, extended by zero, is smooth at both
$q$ and $\partial B_g(q,R)$.

For $r>0$,
\begin{equation}\label{eq:hess-radial-w}
 \Hess_g w=q_1(r)dr^2+Q(r)\Hess_g r.
\end{equation}
The main tangential term $Q(r)r^{-1}(g-dr^2)$ is nonpositive.  The error in \eqref{eq:hess-r} has norm at most
$C|Q|r\le C C_1\sigma R^2$.  Decrease $R_*$ so that this is at most $3\sigma/4$.  Equations
\eqref{eq:q-properties} and \eqref{eq:hess-radial-w} then give
$\Hess w\le\sigma g$, while their lower bounds give
$\Hess w\ge-C_*g$ for a fixed $C_*$.  On $r\le\ell$,
\[
 \Hess_gw=-A\Hess_g(r^2/2).
\]
After one more decrease of $R_*$, this is at most $-3g$.
Thus one may take $\rho=\ell=c\sigma R$.  Finally,
\eqref{eq:Q-properties} gives
\[
 |dw|=|Q|\le C_*\sigma R,
 \qquad
 0\le w\le R\max|Q|\le C_*\sigma R^2.
\]
\end{proof}

\section{Construction before the final pullback}

Let $N\in S^2$ be the north pole and set
\[
 q=(N,N)\in M.
\]
Apply Lemma~\ref{lem:needle} to $(M,g_0,q)$, and retain its constants
$R_*,c_*,C_*$.  Decrease $R_*$ so that $B_{g_0}(q,R_*)$ is disjoint from
$\Sigma_\delta$ for every $0<\delta\le1/100$.

Choose a small constant $\tau_0\in(0,1)$, to be fixed below.  Given
$\varepsilon>0$, put
\begin{equation}\label{eq:tau}
 \tau=\min\{\varepsilon,\tau_0\},
 \qquad
 \delta=\frac{\tau}{100},
 \qquad
 a=\frac{\tau\delta}{100(C_\beta+1)},
 \qquad
 \sigma=\frac{\tau}{100(C_*+1)}.
\end{equation}
Choose $R=R_\tau\le\min\{R_*,1\}$ so small that
\begin{equation}\label{eq:R-volume}
 \frac{\Vol(B_{g_0}(q,R),g_0)}{\Vol(M,g_0)}<\frac{\tau}{100}.
\end{equation}
Let $w=w_{\sigma,R}$ be given by Lemma~\ref{lem:needle}, and let
$\rho\ge c_*\sigma R$ be the radius in part (iii) of that lemma.  Define
\begin{equation}\label{eq:f-total}
 f=u_{a,\delta}+w,
 \qquad
 \widehat g_\tau=e^{2f}g_0.
\end{equation}
Both summands are nonnegative.  By decreasing $\tau_0$ if necessary, the preceding estimates give, uniformly in $\tau$,
\begin{equation}\label{eq:f-small}
 0\le f\le\frac1{10},
 \qquad
 |df|_{g_0}\le\frac{\tau}{20}.
\end{equation}
Furthermore, \eqref{eq:u-global}, Lemma~\ref{lem:needle}, and
\eqref{eq:tau} imply
\begin{equation}\label{eq:hess-total-upper}
 \Hess_{g_0}f
 \le\frac{\tau}{50}g_0
\end{equation}
and
\begin{equation}\label{eq:hess-total-lower}
 \Hess_{g_0}f\ge-(C_*+1)g_0.
\end{equation}

\begin{proposition}\label{prop:curvature-hat}
After choosing $\tau_0$ sufficiently small, the metrics
$\widehat g_\tau$ satisfy the following estimates.
\begin{enumerate}[label=\textup{(\roman*)}]
 \item On all of $M$,
 \[
  -\tau<\sec_{\widehat g_\tau}\le\Lambda,
  \qquad
  \Lambda:=2C_*+4.
 \]
 \item On $B_{g_0}(q,\rho)$, every sectional curvature is greater than $1$.
 \item On
 \[
  M\setminus\bigl(\Sigma_\delta\cup\overline{B_{g_0}(q,R)}\bigr),
 \]
every sectional curvature is strictly positive.
\end{enumerate}
\end{proposition}

\begin{proof}
For a $g_0$-orthonormal basis $e_1,e_2$ of a plane, equations
\eqref{eq:conformal-curvature}, \eqref{eq:product-bounds},
\eqref{eq:f-small}, and \eqref{eq:hess-total-upper} give
\[
 e^{2f}\sec_{\widehat g_\tau}(P)
 \ge-\frac{2\tau}{50}-\frac{\tau^2}{400}>-\tau.
\]
Since $f\ge0$, multiplication by $e^{-2f}$ cannot make this lower bound more negative.  The lower Hessian estimate gives
\[
 e^{2f}\sec_{\widehat g_\tau}(P)
 \le1+2(C_*+1)+\frac{\tau^2}{400}
 \le2C_*+4,
\]
after decreasing $\tau_0$.

On $B_{g_0}(q,\rho)$, Lemma~\ref{lem:needle} and
\eqref{eq:u-global} give
\[
 \Hess_{g_0}f
 \le\left(-3+\frac{\tau}{100}\right)g_0.
\]
Thus
\[
 e^{2f}\sec_{\widehat g_\tau}(P)
 \ge 2\left(3-\frac{\tau}{100}\right)-\frac{\tau^2}{400}.
\]
Together with $f\le1/10$, this is greater than $1$ for small
$\tau_0$.

Finally, outside $\Sigma_\delta\cup\overline{B_{g_0}(q,R)}$ one has
$w=0$.  Proposition~\ref{prop:base-conformal} applies because
\[
 \frac a\delta=\frac{\tau}{100(C_\beta+1)}<1,
\]
and gives strict positivity there.
\end{proof}

Define the closed set
\begin{equation}\label{eq:Shat}
 \widehat S_\tau
 =\Sigma_\delta\cup
 \left(\overline{B_{g_0}(q,R)}\setminus B_{g_0}(q,\rho)\right).
\end{equation}
Proposition~\ref{prop:curvature-hat} shows that
\begin{equation}\label{eq:positive-away-hat}
 \sec_{\widehat g_\tau}>0
 \quad\text{on }M\setminus\widehat S_\tau,
\end{equation}
and all sectional curvatures on the core $B_{g_0}(q,\rho)$ are greater than $1$.

By \eqref{eq:S0-volume}, \eqref{eq:R-volume}, and
$0\le f\le1/10$,
\begin{align}
 \frac{\Vol(\widehat S_\tau,\widehat g_\tau)}
      {\Vol(M,\widehat g_\tau)}
 &\le e^{4/10}
 \frac{\Vol(\widehat S_\tau,g_0)}{\Vol(M,g_0)}\notag\\
 &<e^{4/10}\left(\frac{\tau}{50}+\frac{\tau}{100}\right)
 <\tau.\label{eq:relative-volume-hat}
\end{align}

\section{Uniform geometry and the fixed open ball}

The inequalities $0\le f\le1/10$ imply
\begin{equation}\label{eq:bilipschitz}
 g_0\le\widehat g_\tau\le e^{1/5}g_0.
\end{equation}
Thus the volume and diameter of $\widehat g_\tau$ have uniform positive lower and upper bounds.  We also need a uniform injectivity-radius bound.

\begin{lemma}[injectivity from first-order control]\label{lem:inj}
Let $N$ be a closed manifold and $g_j$ a family of smooth metrics.  Assume that
\begin{enumerate}[label=\textup{(\roman*)}]
 \item the $g_j$ are uniformly bilipschitz to one fixed smooth metric;
 \item in one fixed finite coordinate atlas, the Christoffel symbols of the $g_j$ are uniformly bounded;
 \item $\sec_{g_j}\le\Lambda$ for one constant $\Lambda$.
\end{enumerate}
Then $\inj(N,g_j)$ has a uniform positive lower bound.
\end{lemma}

\begin{proof}
By the Rauch comparison theorem, the conjugate radius is bounded below by a positive constant depending only on $\Lambda$.  If the injectivity radii tended to zero, the standard injectivity-radius lemma would therefore produce geodesic loops $\gamma_j:[0,L_j]\to N$ with $L_j\to0$.

Uniform bilipschitz control implies that each sufficiently short loop lies in one chart of the fixed finite atlas.  Parametrize by $g_j$-arclength and write its coordinates as $x_j(t)$.  Uniform ellipticity gives constants $c,C>0$ such that
\[
 c\le|\dot x_j(t)|\le C.
\]
The geodesic equation and the uniform Christoffel bound give
$|\ddot x_j(t)|\le A$ for one constant $A$.  Since
$x_j(L_j)=x_j(0)$,
\[
 0=\int_0^{L_j}\dot x_j(t)\,dt
 =L_j\dot x_j(0)
   +\int_0^{L_j}(L_j-t)\ddot x_j(t)\,dt.
\]
Hence
\[
 |\dot x_j(0)|\le\frac A2L_j,
\]
contradicting the lower bound $c$.
\end{proof}

For a conformal metric $e^{2f}g_0$, the Levi-Civita connections are related by
\begin{equation}\label{eq:connection-conformal}
 \widetilde\nabla_XY
 =\nabla^0_XY+df(X)Y+df(Y)X-g_0(X,Y)\nabla^0f.
\end{equation}
Thus \eqref{eq:f-small} gives a uniform Christoffel bound in a fixed atlas.  Lemma~\ref{lem:inj}, Proposition~\ref{prop:curvature-hat}, and
\eqref{eq:bilipschitz} therefore give
\begin{equation}\label{eq:inj-hat}
 \inj(M,\widehat g_\tau)\ge i_0>0
\end{equation}
for one constant $i_0$ independent of $\tau$.

It remains to place the uniformly positively curved core over one fixed open ball.  Choose once and for all a small $g_0$-geodesic ball
\[
 U=B_{g_0}(q,r_U)
\]
whose closure is contained in a normal coordinate neighborhood of $q$.  For every $\tau$, there is a diffeomorphism
\[
 F_\tau:M\longrightarrow M
\]
such that
\begin{equation}\label{eq:compress-U}
 F_\tau(U)\subset B_{g_0}(q,\rho/2).
\end{equation}
Indeed, in polar coordinates about $q$, use a smooth increasing radial map that sends $[0,r_U]$ into $[0,\rho/2]$ and agrees with the identity outside a slightly larger fixed coordinate ball.

For the original parameter $\varepsilon$, with $\tau$ as in
\eqref{eq:tau}, set
\begin{equation}\label{eq:final-def}
 g_\varepsilon=F_\tau^*\widehat g_\tau,
 \qquad
 S_\varepsilon=F_\tau^{-1}(\widehat S_\tau).
\end{equation}
The map
\[
 F_\tau:(M,g_\varepsilon)\longrightarrow(M,\widehat g_\tau)
\]
is an isometry.  Therefore all sectional-curvature bounds, relative volumes, total volumes, diameters, and injectivity radii are unchanged.  Equations
\eqref{eq:compress-U} and \eqref{eq:Shat} also give
$U\cap S_\varepsilon=\varnothing$, while Proposition~\ref{prop:curvature-hat} gives
$\sec_{g_\varepsilon}>1$ on $U$.

Taking, for example,
\[
 v_-=\Vol(M,g_0),\qquad
 v_+=e^{2/5}\Vol(M,g_0),
\]
\[
 d_-=\diam(M,g_0),\qquad
 d_+=e^{1/10}\diam(M,g_0),
\]
and using \eqref{eq:relative-volume-hat},
\eqref{eq:inj-hat}, and Proposition~\ref{prop:curvature-hat}, proves
Theorem~\ref{thm:main}.  Notice that when $\varepsilon\ge\tau_0$, one has $\tau=\tau_0$; then
$-\tau_0> -\varepsilon$ and the relative-volume estimate
$<\tau_0\le\varepsilon$ still gives the required assertions.

\section{Why this does not produce a quasi-positively curved limit}

The construction exploits two freedoms that are absent from a genuine compactness argument.
First, the set $\Sigma_\delta$ is thin in volume but spread across the manifold.  It intersects the closed directions that obstruct the existence of a globally strictly concave function, while its relative volume tends to zero.  The positive curvature created outside it is very small, of order $a\delta$.

Second, the region on which all sectional curvatures exceed $1$ has intrinsic size tending to zero before the final pullback.  The diffeomorphism $F_\tau$ places the fixed topological ball $U$ over that shrinking core without changing any intrinsic global bound.  Consequently the final metrics need not be precompact in $C^2$ in the fixed differentiable gauge.  The family therefore does not yield a quasi-positively curved limit metric on $S^2\times S^2$.

\end{document}